\documentclass{amsart}
\usepackage{amscd}
\usepackage{amssymb}
\input xy
\xyoption{all}

\newcommand{\comment}[1]{}

\def\S{{{\mathbb S}^1}}
\def\ad#1{{\scriptstyle \left[ \frac{1}{#1} \right]}}

\def\Ncy{N^{\mathrm{cy}}}
\def\tNcy{\widetilde{N}^{\mathrm{cy}}}

\def\smsh{\wedge}

\def\cdh{{\mathrm{cdh}}}

\def\zar{{\mathrm{zar}}}

\def\triqui{\vartriangleleft}

\def\oo{\otimes}

\newcommand{\Q}{\mathbb{Q}}

\newcommand{\bH}{{\mathbb{H}}}

\newcommand{\C}{\mathbb{C}}
\newcommand{\Z}{\mathbb{Z}}

\newcommand{\perf}{\operatorname{perf}}

\def\cA{\mathcal A}
\def\cB{\mathcal B}
\def\cC{\mathcal C}

\def\cE{\mathcal E}
\def\cF{\mathcal F}

\def\cI{\mathcal I}

\def\cO{\mathcal O}

\def\cMpctf{{\mathcal M}_{\text{pctf}}}

\def\sn{\mathrm{sn}}
\def\cK{\mathcal K}
\def\cKH{\mathcal KH}

\def\Spec{\operatorname{Spec}}
\def\Sch{\operatorname{Sch}}
\def\MSpec{\operatorname{MSpec}}
\def\map#1{\, {\buildrel #1 \over \lra}\, }

\def\lmap#1{\,{\buildrel #1 \over \longleftarrow}\,}
\def\lra{\longrightarrow}

\def\onto{\twoheadrightarrow}

\def\tOmega{\widetilde{\Omega}}

\DeclareMathOperator*{\colim}{colim}

\newcommand{\fc}{{\mathfrak c}}
\newcommand{\fp}{{\mathfrak p}}
\newcommand{\fm}{{\mathfrak m}}

\numberwithin{equation}{section}

\theoremstyle{plain} 
\newtheorem{thm}[equation]{Theorem}
\newtheorem{cor}[equation]{Corollary}
\newtheorem{lem}[equation]{Lemma}
\newtheorem{prop}[equation]{Proposition}
\newtheorem{substuff}{Remark}[equation]

\newtheorem{DilationThm}[equation]{Dilation Theorem}

\theoremstyle{definition}
\newtheorem{defn}[equation]{Definition}

\newtheorem{ex}[equation]{Example}

\theoremstyle{remark}
\newtheorem{rem}[equation]{Remark}
\newtheorem{subrem}[substuff]{Remark} 
\newtheorem{subex}[substuff]{\bf Example} 

\newtheorem{notation}[equation]{Notation}

\begin{document}

\title[$K$-theory of toric schemes in mixed characteristic]
{The $K$-theory of toric schemes over regular rings of mixed characteristic}
\date{\today}

\author{G. Corti\~nas}
\thanks{Corti\~nas' research was supported by Conicet
and partially supported by grants UBACyT 20021030100481BA, PIP
112-201101-00800CO, PICT 2013-0454, and MTM2015-65764-C3-1-P (Feder funds).}
\address{Dept.\ Matem\'atica-Inst.\ Santal\'o, FCEyN,
Universidad de Buenos Aires,
Ciudad Universitaria, (1428) Buenos Aires, Argentina}
\email{gcorti@dm.uba.ar}

\author{C. Haesemeyer}
\thanks{Haesemeyer's research was partially supported by the University of
  Melbourne Research Grant Support Scheme and ARC Discovery Project grant
  DP170102328} 
\address{School of Mathematics and Statistics, University of Melbourne,
Victoria 3010, Australia}
\email{christian.haesemeyer@unimelb.edu.au}

\author{Mark E. Walker}
\thanks{Walker's research was partially supported by a grant from the
Simons Foundation (\#318705).}
\address{Dept.\ of Mathematics, University of Nebraska - Lincoln,
  Lincoln, NE 68588, USA}
\email{mark.walker@unl.edu}

\author{C. Weibel}
\thanks{Weibel's research was supported by NSA and NSF grants.}
\address{Dept.\ of Mathematics, Rutgers University, New Brunswick,
NJ 08901, USA} \email{weibel@math.rutgers.edu}

\begin{abstract}
We show that if $X$ is a toric scheme over a 
regular commutative ring $k$
then the direct limit of the $K$-groups of $X$ taken over any
infinite sequence of nontrivial dilations is homotopy invariant.
This theorem was previously known for regular commutative rings 
containing a field. The affine 
case of our result was conjectured by Gubeladze.
We prove analogous results when $k$ is replaced by an
appropriate $K$-regular, not necessarily commutative $k$-algebra. 
\end{abstract}



\maketitle

\section*{Introduction}\label{introduction}

Let $A = (A, \cdot)$ be a commutative monoid. For each integer $c\ge2$,
 the $c$-th power map $\theta_c:A\to A$, $\theta_c(a)=a^c$, is
an endomorphism of $A$; it is called a {\it dilation}.  
For any ring $k$, $\theta_c$ induces an endomorphism
of the monoid ring $k[A]$, its $K$-theory $K_*(k[A])$, and its
homotopy $K$-theory $KH_*(k[A])$.

The affine {\it Dilation Theorem} says that the monoid of dilations
acts nilpotently on the reduced $K$-theory of $k[A]$, at least
when $A$ is a submonoid of a torsionfree abelian group,
$A$ has no nontrivial units and $k$ is an appropriately regular ring.  

\begin{DilationThm}\label{Main-intro}
Let $A$ be a submonoid of a torsionfree abelian group.
Assume that $A$ has no non-trivial units. Let $\fc = (c_1,c_2,\dots)$
be a sequence of integers with $c_i \ge2$ for all $i$, and let
$\Lambda$ be either
\begin{enumerate}
\item[a)] a regular commutative ring; 
\item[b)] a commutative $C^\ast$-algebra;
\item[c)] an associative regular ring that admits the structure of a  flat algebra over a 
regular commutative ring $k_0$ of finite Krull dimension; 
\item[d)] an associative ring that is $K$-regular and that admits the structure of a flat algebra over a 
regular commutative ring $k_0$ of finite Krull dimension having infinite residue fields; or
\item[e)] an associative ring that is $K$-regular and that admits the  structure of a flat algebra over a  regular commutative ring $k_0$ of
finite Krull dimension such that for every prime $\fp\subset k_0$ having finite residue field, the ring $\Lambda\oo_{k_0}k'$ is $K$-regular for every
\'etale $(k_0)_{\fp}$-algebra $k'$.
\end{enumerate}
Then the canonical map is an isomorphism:
\begin{equation}\label{map:intromain}
K_*(\Lambda) \map{\cong} \varinjlim\nolimits_{\theta_{c}} K_*(\Lambda[A]).
\end{equation}
\end{DilationThm}

Recall that an associative ring $\Lambda$ is called (right) \emph{regular} if it is (right) Noetherian and every finitely generated
(right) $\Lambda$ module has finite projective dimension. It is
called {\em   $K$-regular} if the canonical map $K_*(\Lambda) \to K_*(\Lambda[x_1, \dots, x_n])$ is an isomorphism for all $n \geq 0$.      
Also recall that, up to canonical isomorphism, a (unital) commutative $C^*$-algebra is the ring $C(T)$ of complex-valued continuous functions on a compact Hausdorff space $T$. All rings
 considered in this article are unital; in particular, nonunital $C^*$-algebras are not considered.

Case (a) of Theorem \ref{Main-intro} verifies a conjecture of Gubeladze (\cite[1.1]{Gu05}).

Gubeladze verified the conjecture for $\Q$-algebras in \cite{Gu05} 
and \cite{Gu08}; the conjecture was verified for $k$ of positive 
characteristic by the authors in \cite{chww-p}. Gubeladze also verified the 
conjecture when $A$ is a ``simplicial'' monoid for all commutative regular rings $k$
 in \cite{Gu03}. Further, his result \cite[Theorem 3.2.2]{Gu94} implies that if $A$ is simplicial and $k$ is a not necessarily commutative $K$-regular ring, then the map
 \eqref{map:intromain}  is an isomorphism whenever the sequence $\fc$ is constantly equal to a prime number.

As we showed in \cite{chww-monoid}, it is useful to pass from
abelian monoids to pointed abelian monoids (adding an element `0'),
and to generalize even further to {\it monoid schemes}, associating the
affine monoid scheme $\MSpec(A)$ to a pointed monoid $A$.
A monoid scheme $X$ has a {\it $k$-realization} $X_k$ over any
commutative ring $k$; if $X=\MSpec(A)$ then $X_k$ is $\Spec(k[A])$.
Again there are dilation maps $\theta_c: X\to X$, defined
locally as the $c$-th power map on affine open subschemes, and $\theta_c$
induces dilations of both $X_k$ and its $K$-theory.

Even if $\Lambda$ is a noncommutative ring, we can still make sense
out of $K(X_\Lambda)$, although the scheme $X_\Lambda$ is not defined. If $X=\MSpec(A)$ then $K(X_\Lambda)$ is 
$K(\Lambda[A])$; in general, the spectrum $K(X_\Lambda)$ may either be defined 
using Zariski descent on $X$,
or equivalently as the $K$-theory of the dg category $\perf(X_k)\oo_k^L\Lambda$;
see Example \ref{ex:KHC} for details. 

Theorem \ref{intro-Dilation} below is the {\it Dilation Theorem} for
monoid schemes; the hypotheses on $X$ are satisfied whenever it is a toric
monoid scheme or, more generally, a partially cancellative torsionfree (pctf)
monoid scheme of finite type. (The definitions of all these
terms are recalled below, at the end of this introduction.) They are
also satisfied by $X=\MSpec(A)$ when $A$ is a monoid of finite type
satisfying the assumptions  in Theorem \ref{Main-intro}. 
Observe also
that since every abelian monoid is the filtered colimit of its
finitely generated submonoids, and since $K$-theory commutes with such
colimits, the finitely generated case of Theorem \ref{Main-intro}
implies the general case.  We remark also that for $A$ and $\Lambda$ as in
Theorem \ref{Main-intro}, $K_*(\Lambda)=KH_*(\Lambda)=KH_*(\Lambda[A])$. 
Thus Theorem \ref{Main-intro} follows from Theorem \ref{intro-Dilation}.

\begin{thm}\label{intro-Dilation}
Let $X$ be a separated, partially cancellative, torsionfree monoid scheme of
finite type, let $\fc=(c_1,c_2,\dots)$ be a sequence of integers with $c_i \ge2$ for all $i$ and
let $\Lambda$ be an associative ring satisfying one of (a) through (e) in Theorem \ref{Main-intro}. 
Then the canonical map is an isomorphism:
\begin{equation}\label{map:dilate}
\varinjlim\nolimits_{\,\theta_{c}} K_*(X_\Lambda) \map{\cong}
\varinjlim\nolimits_{\,\theta_{c}} KH_*(X_\Lambda).
\end{equation}
\end{thm}

The particular case of this theorem when $\Lambda$ is a commutative 
regular ring containing a field was proven
by the authors in \cite{chww-toric} and \cite{chww-p}, and was used to verify
Gubeladze's conjecture for these rings.  The proof of the dilation theorem
presented in this paper follows the strategy used in \cite{chww-p}. The basic
outline is as follows:

\noindent{\bf Step 1.} We show that the singularities of any
nice monoid scheme may be resolved by a finite sequence of blow-ups
along smooth normally flat centers. This is based upon a theorem of
Bierstone-Milman (see \cite{BM}).
Even if we begin with a monoid, the blow-ups of its monoid scheme
take us out of the realm of (affine) monoids. This step  is carried out in
Section \ref{sec:nf}, and is preceded by a short calculation with monoids in
Section \ref{sec:rivalfree}. This step is independent of $\Lambda$.

\noindent {\bf Step 2.} In \cite{chww-monoid}, we introduced the notion of
$cdh$ descent for presheaves on pctf monoid schemes, and showed that the
functor $X\mapsto KH(X_\Lambda)$ from pctf monoid schemes to spectra satisfies
$cdh$ descent when $\Lambda$ is a commutative regular ring 
of finite Krull dimension that contains a field.
In Section \ref{sec:KHdescent}, we review and use Step\,1 to modify this
proof, replacing the assumptions on $\Lambda$ given above by the assumptions that $\Lambda$ is a 
flat associative algebra over a commutative regular ring of finite Krull dimension all of whose
residue fields are infinite; see
Corollaries \ref{cor:khdescent} and \ref{cor:khCdescent}.  
Also in this section, we apply the arguments
of \cite{chww-p} to prove that both the fiber of the Jones-Goodwillie
Chern character $K(X_\Lambda)\otimes\Q\to HN(X_{\Lambda\otimes\Q})$ 
and the fiber of the $p$-local cyclotomic trace with $\Z/p^n$-coefficients to
the pro-spectrum $\{TC^\nu(X_\Lambda;p)\}_\nu$ satisfy $cdh$ descent whenever
$\Lambda$ is as in Theorem \ref{Main-intro}
(see Propositions \ref{Qfiber} and \ref{pfiber}).

\noindent{\bf Step 3.} For a presheaf of spectra $E$ on monoid schemes, 
write $\cF_E$ for the homotopy fiber of the map from $E$ to its $cdh$-fibrant
replacement. From Step ~2 we conclude that if $\Lambda$ is a flat algebra over a
commutative regular ring $k_0$ of finite Krull dimension all of whose residue fields are 
infinite, then $\cF_K$ is equivalent to the fiber of the map $K\to KH$, its
rationalization is equivalent to $\cF_{HN(-\otimes\Q)}$, and for each $n$ the
cofiber $\cF_K/p^n$ of multiplication by $p^n$ is equivalent to
the pro-spectrum
$\{ \cF_{E_\nu} \}$ for $E_\nu = TC^\nu(-;p)/p^n$.
Making use of the constructions developed in
\cite[Section 6]{chww-p}, we prove that for an arbitrary ring $\Lambda$,
monoid scheme $X$ and $\nu\ge 1$, 
both $\cF_{HN(-\otimes\Q)}(X)$ and $\cF_{TC^\nu(-;p)}(X)$ 
become contractible upon taking the colimit over any infinite sequence of
non-trivial dilations. The case of $TC^\nu(-;p)$ is considered in Section
\ref{sec:dilatetc} (see Theorem \ref{thm:TCdilated}) and that of $HN(-\oo\Q)$
in Section \ref{sec:dilatehn} (see Corollary \ref{cor:HNdilated}).

\noindent{\bf Step 4.} Theorem \ref{intro-Dilation} is finally obtained in
Section \ref{sec:reduction}:   Theorem \ref{thm:main}
gives parts (c) and (e) and the other parts are consequences. 
The particular case when $\Lambda$ 
is commutative and contains a field is \cite[Theorem 8.3]{chww-p}. 
The case when $\Lambda$ is flat over a commutative regular ring $k_0$ of finite Krull dimension
all of whose residue fields are infinite follows using Steps~2-3 and 
the argument of the proof of \cite[Theorem 8.3]{chww-p}. We
show further that, if  $k_0$ has finite 
Krull dimension and $\Lambda\supset k_0$ satisfies the hypothesis of the theorem, then the map \eqref{map:dilate} is an isomorphism 
if and only if this happens for every local ring of $k_0$ 
(Lemma \ref{lem:reducetolocal}). Then we show
that for $k_0$ a regular local ring, the case when the residue field is
infinite implies the case of finite residue field. This establishes parts (c) and (e) of the theorem. 
Part (d) is a special case of part (e) and part (b) follows from 
the fact that commutative $C^*$-algebras are $K$-regular (\cite[Theorem 8.2]{CortinasThom}). 
Finally we observe  
that every commutative regular ring $k$ with 
$\Spec(k)$ connected is a flat extension of either a 
finite field or a localization of the ring of integers; this proves (a).

\goodbreak

\bigskip

\noindent{\bf Some notation:} Given a presheaf $E$ from some full
subcategory of the category of monoid schemes to spectra (or spaces,
or chain complexes, or equivariant spectra or spaces) and a sequence
$\fc = (c_1,c_2,c_3,\dotsc)$ of integers larger than $1$, we
write $E^\fc$ for the (sequential) homotopy colimit
$\varinjlim\nolimits_{\,\theta_{c_i}}\!E$, again a presheaf of spectra
(or spaces, or ...) on the same category of monoid schemes. If our
presheaf is obtained by sending the monoid scheme $X$ to the spectrum
(or space, etc.\!) $E(X_\Lambda)$ for a ring $\Lambda$ we will sometimes 
(when no confusion is possible) abuse notation and
write $E^\fc$ for the presheaf 
$X\mapsto\varinjlim\nolimits_{\,\theta_{c_i}}\!E(X_\Lambda)$.

\noindent{\bf Monoid terminology:} 
Unless otherwise stated, a {\em monoid} is a pointed
abelian monoid written multiplicatively, i.e., it is an abelian
monoid with unit\,1 and a basepoint $0$ satisfying $a\cdot0=0$ 
for all $a\in A$.  If $B$ is an unpointed monoid, adjoining an element
$0$ yields a pointed monoid $B_+$. If $A$ is pointed monoid, we write $A'$ for the subset $A\setminus\{0\}$. 
We say that $A$ is {\em cancellative} if
$ab=ac$ implies $b=c$ or $a=0$ for any $a,b,c\in A$.
In particular, if $A$ is cancellative, then 
$A'$ is an unpointed submonoid of $A$. 
We say $A$ is {\em torsionfree} if whenever $a^n=b^n$ for $a,b\in A$ and
some $n\ge1$, we have $a=b$. A monoid is cancellative and torsionfree
if and only if it is a submonoid of the pointed monoid $T_+$ associated to a
torsionfree abelian group $T$.

An {\em ideal} $I$ of a monoid $A$ is a subset containing $0$ and satisfying 
$AI\subseteq I$. In this case, the quotient monoid $A/I$ is defined
by collapsing $I$ to $0$; the product in $A/I$ is the unique one 
making the canonical surjection $A \onto A/I$ a morphism.
A monoid $A$ is said to be {\it partially cancellative} if it is
a quotient $C/I$ of a cancellative monoid $C$ by some ideal $I$; 
if $C$ is both cancellative and torsionfree, we say that $A$ is {\it pctf}.

The prime ideals in $A$ form a space $\MSpec(A)$, with a sheaf of monoids. 
A {\it monoid scheme} is a space $X$ with a sheaf of monoids $\cA$
that is locally isomorphic to $\MSpec(A)$ for some $A$.  
A closed subscheme $Z$ is {\it equivariant} if its structure sheaf
is $\cA/\cI$ for a sheaf of ideals $\cI$.

\goodbreak
\section{Free $A$-sets} \label{sec:rivalfree}

Given a monoid $A$, an {\em $A$-set} is a pointed set $X$, with
basepoint $0$, together with a function $A \smsh X \to X$, written
$(a,x) \mapsto a \cdot x$, such that $a \cdot (a' \cdot x) = aa' \cdot
x$ for all $a,a' \in A$ and $x \in X$. The $k$-realization $k[X]$ of
an $A$-set is a $k[A]$-module. The monoid $A$ is an $A$-set in an obvious way,
and given any collection of $A$-sets, their wedge sum is again an
$A$-set.  A subset $B$ of an $A$-set $X$ is a {\em basis} if $X=\bigvee_{b\in B} Ab$ and if $ab = a'b'$, for $a,a' \in A'$ and $b,b' \in B$, implies that either 
$b =b'$ and $a = a'$ or $a=a'=0$. We say that an $A$-set $X$ is {\em free} if it has a basis, or equivalently if it is isomorphic
to a wedge sum of copies of $A$, $\bigvee_{i \in I} A$. 

The goal of this section is to prove the following result.

\begin{thm}\label{rivalthm:free A-set}
Let $A$ be a cancellative monoid, $X$ a finitely
generated $A$-set and $k$ a commutative ring.  If $k[X]$ is a free
$k[A]$-module, then $X$ is a free $A$-set.
\end{thm}

\begin{ex}\label{ex:G-sets}
Suppose that $A=G_+$ for an abelian group $G$.  Then $X\setminus\{0\}$
is a disjoint union of orbits $G/H_i$, and $k[X]$ is the direct sum of
the $k[G/H_i]$.  If $k[X]$ is a free $k[A]$-module, then each $H_i$
is trivial, so $X$ is a free $A$-set.
\end{ex}

The proof requires a sequence of preliminary results.

Let $A^\times$ denote the group of units of $A$ and define the
quotient monoid $\overline{A} = A/\sim$ where $a \sim b$ if and only
if $a = ub$ for some $u \in A^\times$. Similarly, 
$\overline{X} = X\smsh_A\overline{A}$ is the $\overline{A}$-set
$X/\sim$, where $x \sim y$ if and only if $x = uy$ for some $u\in A^\times$. 

\begin{lem}\label{lem:barA}
For any monoid $A$ and any commutative ring $k$, 
$k[\overline{A}] \cong k[A]\oo_{k[A^\times]}k$.
If $k[X]$ is free as a $k[A]$-module, then 
$k[\overline{X}]$ is free as a $k[\overline{A}]$-module.
\end{lem}

\begin{proof}
The kernel $I$ of $k[A^\times]\to k$ is generated by $\{u-1:
u\in A^\times\}$; we have to show that the surjection
$k[A]/Ik[A]\to k[\overline{A}]$ is an isomorphism.
Choose a set of representatives $\{a_i\}$ for $\sim$, so that
$A$ is the disjoint union of $0$ and the sets $A^\times a_i$.
Then $k[A]$ is the direct sum of the $k[A^\times a_i]$,
and $k[A]/Ik[A]$ is the direct sum of the $k[a_i]$, i.e.,
$k[\overline{A}]$.  The result for free modules follows
from the equation below (see \cite[paragraph below diagram 1.8]{chww-monoid})
\[ 
k[\overline{X}]=k[X\smsh_A\overline{A}] \cong
k[X] \oo_{k[A]} k[\overline{A}].\qedhere\]
\end{proof}

We will say that an $A$-set $X$ is {\em cancellative} if
$ax=bx$ implies $a=b$ for any nonzero $x\in X$ and $a,b\in A$.

\begin{lem} \label{lem1-7}
Assume $A$ is a cancellative monoid and $X$ is a cancellative $A$-set.
Then $X$ is free as an $A$-set if and only if  $\overline{X}$ is free 
as an $\overline{A}$-set.
\end{lem}

\begin{proof}
The forward direction holds since 
$\overline{\bigvee_{I} A} = \bigvee_{I} \overline{A}$.
Conversely, suppose that $\overline{X}$ is free as an 
$\overline{A}$-set with basis $\overline{B}$.  Let $B$ be a subset of $X$
given by choosing one representative for each element of
$\overline{B}$.  It is clear that $B$ generates $X$, and is an
$A$-basis of $X$ because $X$ is cancellative: 
if $ab=a'b$ in $X$ then $a=a'$.
\end{proof}

\begin{rem}
In Lemma \ref{lem1-7}, the hypothesis that $X$ be cancellative is needed, as
the example $A=\{0,1,-1\}$ and $X=\{0,1\}$ shows; in this case
$\overline{A}=\overline{X}$.
\end{rem}

\begin{rem}\label{poset}
If $A$ is cancellative and has no non-trivial units, 
we  define a partial ordering on $A$ by $b\le c$ if $c=ab$ for
some $a$ in $A$. This partial ordering respects the group
operation on $A$: if $a\le b$ then $ac\le bc$ for all $c \in A$. 

The group completion $G_+$ of such a monoid $A$ is also partially ordered: $h\le g$
if $g=ah$ for some $a$ in $A$. The inclusion $A\subset G_+$
respects the partial ordering.
\end{rem}

We say that a ring $R$ is {\it $A$-graded} if $R = \bigoplus_{a\in A} R_a$ 
with $R_aR_b\subseteq R_{ab}$ and $R_0=\{0\}$. 
If $A$  has no non-trivial units, then
$R_{\ne 1}=\bigoplus_{a\ne1} R_a$ is an ideal of $R$, with $R_1=R/R_{\ne 1}$. 

\begin{ex}\label{ex:kAgraded}
For any monoid $A$ and ring $k$, the monoid ring $k[A]$ is $A$-graded with
$k[A]_a := k \cdot a$. If $A$ has no non-trivial units,
$k[A]_{\ne1}=k[\fm]$ where $\fm=\fm_A$ is the
unique maximal ideal of $A$.  
\end{ex}

\begin{lem} \label{lem1-7b}
Let $A$ be a  
cancellative monoid with no non-trivial units.
Suppose $R$ is an $A$-graded ring.  
If $T$ is a finitely generated nonzero $A$-graded $R$-module such that $T_0 = 0$,
then $R_{\ne 1}T \ne T$.
\end{lem}

\begin{proof} 
Set $S = \{a\in A': T_a \ne 0\}$, and suppose
that $T$ is generated by homogenous elements $x_1,\dots,x_l$
of degrees $s_1,\dots,s_l$ ($s_i\in S$).
Let $S_0$ denote the set of minimal elements in $\{s_1,\dots,s_l\}$ with respect to the partial ordering in Remark \ref{poset}. 
Then every element of $S_0$ is minimal in $S$.
If $s\in S$, there is a non-zero homogeneous element $t$ in $T_s$. 
Write $t=\sum_i r_i x_i$. Grouping the sum by common degrees, we may 
assume the $r_i$ are homogeneous and that $r_i = 0$ unless $s=|r_ix_i|$.
Since $t\ne0$, there is an $i$ so that $s=|r_ix_i|\ge s_i$.
This shows that $s \geq s_0$ for some $s_0 \in S_0$.

Fix $s \in S_0$ and let $t$ be a non-zero element of $T_s$. If 
$t\in R_{\ne 1}T$ then $t = \sum_i r_i x_i$ with $r_i \in R_{\ne 1}$.
As before, we may assume all the $r_i$ are homogenous and nonzero,
and that $r_i x_i \in R_s$ for all $i$. Since $r_i \in R_{\ne 1}$, 
we have $|r_i| > 1$ and so $|r_i x_i| > 1 |x_i|=s_i$.
This contradicts the minimality of $s$, showing that $T\ne R_{\ne 1}T$.
\end{proof}

\begin{cor} \label{cor1-7} 
Let $A$ be a cancellative monoid with no nontrivial units.
Suppose $R$ is an $A$-graded ring and $T$ is a finitely generated, free graded
$R$-module. If $Y$ is any set of homogeneous elements in $T$ 
whose image in $T/R_{\ne 1}T$ is a basis of the free $R_1$-module
$T/R_{\ne 1}T$, then $Y$ is a basis of $T$ as an $R$-module.
\end{cor}

\begin{proof} The classical proof applies; here are the details:

First note that $Y$ must be finite, since it maps mod $R_{\ne 1}$ to a basis of a free, finitely generated $R_1$-module.
Let $L$ denote the free graded $R$-module on the set $Y$. The
canonical map $\pi:L\to T$ is graded and $L_0 = T_0 = 0$;
hence its kernel $K$ and cokernel $C$ are both graded modules, and $K_0 = C_0 = 0$.
Since applying $- \otimes_R R/R_{\ne 1}$ to this map results in an
  isomorphism by assumption, $C/R_{\ne 1}C = 0$. By Lemma \ref{lem1-7b}, 
$C = 0$ and hence $\pi$ is surjective. Since $T$ is free graded, 
$\pi$ splits (in the graded category). In particular, $K$ is a direct summand of a finitely generated free module and hence finitely generated. It follows that $K/R_{\ne 1}K = 0$
as well, whence $K=0$ using the Lemma again.
\end{proof}

\begin{lem}\label{lem:Xlocalize}
Suppose $A$ is a monoid, $X$ is an $A$-set and $S\subseteq A$ is a multiplicatively closed subset. Then the natural homomorphism of $k[A]$-modules $S^{-1} k[X] \to k[S^{-1} X]$ is an isomorphism.
\end{lem}

\begin{proof}
The inverse homomorphism is obtained from the natural map of $A$-sets $S^{-1} X \to S^{-1} k[X]$ by adjunction. 
\end{proof}

\begin{lem}\label{lem:Xcancels}
If $A$ is a cancellative monoid and $X$ an $A$-set so that $k[X]$ is a free $k[A]$-module, then $X$ is cancellative.
\end{lem}

\begin{proof}
Consider the group completion of $A$, which is isomorphic to
$G_+$ for some abelian group $G$. Since $k[A]\subseteq k[G]$ and $k[X]$ is free, we have $k[X] \subseteq k[X]\otimes_{k[A]} k[G]$. By Lemma \ref{lem:Xlocalize} and \cite[Lemma 5.5]{chww-monoid}, $k[X]\otimes_{k[A]} k[G]  = k[X_G]$, where $X_G$ is the localization of $X$ at $A'$.
It follows that $X\subseteq X_G$, and the latter is a free $G_+$-set by Example \ref{ex:G-sets}.
Thus, $X$ is a cancellative $A$-set as asserted. 
\end{proof}

\begin{proof}[Proof of Theorem \ref{rivalthm:free A-set}]
By Lemma \ref{lem:Xcancels}, $X$ is a cancellative $A$-set. 
Thus Lemmas \ref{lem:barA} and \ref{lem1-7} imply that we may assume that $A$ has no
non-trivial units.

Define a partial ordering on $X$ by $x\le y$ if and only if $y=ax$ 
for some $a \in A$.  The reflexive and transitive properties are
clear. Since $X$ is cancellative, it is also symmetric:
if $x\le y$ and $y\le x$ there are $a,b\in A$ with $y = ax$, $x = by$
and hence $x = abx$. This implies that $ab=1$; since $A$ has no
nontrivial units, we have $a=b=1$ and hence $x=y$. (When $X = A$, this is the partial ordering of Remark \ref{poset}.)
As $X$ is finitely generated, there is a finite set $B$
of minimal elements of $X$ with respect to the partial ordering. In
fact, $B$ is the unique smallest collection of generators of $X$. We
will prove that $B$ forms a basis of $X$.

Let $F$ denote the free graded $A$-set on $B$. We need to show that the canonical map $F\map{\pi} X$ sending $a[b]$ to $ab$ is a bijection.
The map is surjective since $B$ generates $X$ by construction. 
Passing to $k$-realization gives a surjection of graded $k[A]$-modules,
$k[F]\onto k[X]$.
By Corollary \ref{cor1-7}, with $R=k[A]$ and $Y = B$,
it suffices to prove that this map becomes an isomorphism after 
applying $-\oo_{k[A]} k$, where $k=k[A]/k[A]_{\ne 1}$; this will show that $k[F]\to k[X]$ and hence $F\to X$, are injective. 
By construction, there is a canonical isomorphism 
$k[F]\oo_{k[A]}k\cong k[B_+]$, and, by Example \ref{ex:kAgraded},
$$
k[X] \oo_{k[A]} k = k[X] \oo_{k[A]} k[A/\fm_A] = k[X/\fm_AX].
$$
Since the natural map $B_+\to X/\fm_AX$ induced by $\pi$ is an isomorphism,
the induced map $k[B_+]\to k[X/A_+X]$ is also an isomorphism.
Thus $B$ is a basis of $X$.
\end{proof}

\goodbreak
\section{Normal flatness for monoid schemes}
\label{sec:nf}

Recall from \cite{chww-monoid} that a finitely generated
pointed monoid $A$ is {\it smooth}
if $A$ is the wedge product of a finitely generated free pointed monoid $T$ by $\Gamma_+$,
where $\Gamma$ is a finitely generated free abelian group. That is, for some $d$ and $r$:
\begin{equation*}
A = \langle t_1,\dotsc, t_d, s_1, 1/s_1,\dotsc, s_r, 1/s_r \rangle.
\end{equation*}
Note that $k[A]$ is the Laurent polynomial ring
$k[t_1,\dotsc, t_d, s_1, 1/s_1,\dotsc, s_r, 1/s_r]$ for every ring $k$.

Recall that a commutative
ring $R$ is said to be {\it normally flat} along
an ideal $J$ if each quotient $J^n/J^{n+1}$ is a projective module
over the ring $R/J$. A scheme is said to be normally flat along
a closed subscheme if locally the coordinate ring is normally flat
along the ideal defining it.

\begin{defn}
A monoid $A$ is said to be {\it normally flat} along
an ideal $I$ if each quotient $I^n/I^{n+1}$ is a free $A/I$-set.
Let $X$ be a monoid scheme of finite type with structure sheaf $\cA$;
we say that $X$ is {\it normally flat} along an equivariant closed
subscheme $Z$ if for every $x\in X$ the monoid $\cA_x$
is normally flat along the stalk of the ideal defining $Z$.
\end{defn}

\begin{prop}\label{nflat-I}
Let $I$ be a finitely generated ideal in a monoid $A$, with $A/I$ smooth.
The following are equivalent:
\begin{enumerate}
\item[1)] $A$ is normally flat along $I$.
\item[2)]  $k[A]$ is normally flat along the ideal $k[I]$
 for every commutative ring $k$.
\item[3)] $\Q[A]$ is normally flat along the ideal $\Q[I]$.
\end{enumerate}
\end{prop}

\begin{proof}
The $k$-realization functor commutes
with colimits, since it is left adjoint to the forgetful
functor, so $k[I]^n/k[I]^{n+1}\cong k[I^n/I^{n+1}]$. Hence (1) implies (2). It is clear that (2) implies (3).
Suppose that (3) holds, and write $J$ for $\Q[I]$.
Then each quotient $J^n/J^{n+1} \cong \Q[I^n/I^{n+1}]$ is a projective module over
the ring $\Q[A]/J=\Q[A/I]$, and, 
since all projective $\Q[A/I]$-modules are free \cite{Swan}, 
Theorem \ref{rivalthm:free A-set} implies that each $A/I$-set
$I^n/I^{n+1}$ is free. Hence (1) holds.
\end{proof}

\begin{cor}\label{nflat-Z}
Let $Z$ be a smooth equivariant closed subscheme of $X$,
a monoid scheme of finite type.  The following are equivalent:
\begin{enumerate}
\item[1)] $X$ is normally flat along $Z$.
\item[2)] $X_k$ is normally flat along $Z_k$
 for every commutative ring $k$.
\item[3)] $X_{\Q}$ is normally flat along  $Z_\Q$.
\end{enumerate}
\end{cor}

\begin{proof}
Again, it suffices to prove that (3) implies (1).
As the conditions are local, it suffices to pick $x\in X$
and assume that $X=\MSpec(A)$ and $Z=\MSpec(A/I)$ for
$A=\cA_x$, with $A/I$ smooth.
This case is covered by Proposition \ref{nflat-I}.
\end{proof}

\begin{thm}\label{thm:ros}
Let $X$ be a separated cancellative torsionfree monoid scheme of
finite type, embedded as an equivariant closed subscheme in a 
smooth toric scheme.
Then there is a sequence of blow-ups along smooth equivariant
centers $Z_i\subset X_i$, $0\le i\le n-1$,
\[
Y = X_n\to \dotsm \to X_0 = X
\]
such that $Y$ is smooth  and each $X_i$ is normally flat along $Z_i$.
In addition, for any commutative ring $k$,
each $(X_i)_k$ is normally flat along $(Z_i)_k$.
\end{thm}

\begin{proof}
The case $k=\Q$ of Theorem 14.1 in our paper \cite{chww-monoid}
states that there is a sequence of blowups along smooth equivariant
centers $Z_i\subset X_i$, such that
each $(X_i)_\Q$ is normally flat along $(Z_i)_\Q$.
By Corollary \ref{nflat-Z}, each $X_i$ is normally flat along $Z_i$,
and each $(X_i)_k$ is normally flat along $(Z_i)_k$.
\end{proof}

\goodbreak
\section{A descent theorem for functors via realizations}
\label{sec:KHdescent}

In this section, we recall the notion of $cdh$ descent,
establish a technical result generalizing
\cite[Theorem 14.3]{chww-monoid} to the present context,
and use it to promote several results in {\it op.\ cit.}\ to
our situation.

In more detail, let $\cMpctf$ denote the category of separated pctf monoid schemes of
finite type. Fix a commutative regular ring $k_0$ of finite Krull dimension all of whose residue
fields are  infinite,
and suppose $k_0 \subseteq k$ is a flat extension 
of commutative rings.
We will show that if $F$ is a presheaf of spectra on the category $Sch/k$ of 
separated $k$-schemes essentially of finite presentation
that  satisfies a weak version of $cdh$ descent (see Definition \ref{defn:weakcdh}), then
the presheaf $\cF$ on $\cMpctf$ defined by 
$\cF(X)=F(X_{k})$ satisfies $cdh$ descent; see Theorem \ref{thm:bigdescent}.

We first recall the necessary definitions from \cite{chww-monoid}
and \cite{VVcdh}.
By a {\it cd structure} on a category,
we mean a family of distinguished commutative squares
\begin{equation} \label{sq:cd}
\xymatrix{D \ar[r] \ar[d] & Y \ar[d]^p \\
C \ar[r]^e & X.}
\end{equation}

\goodbreak
\begin{defn}\label{defn:cdh}
A cartesian square \eqref{sq:cd} in $\cMpctf$ is called
\begin{enumerate}
\item
   an {\em abstract blow-up square} if $p$ is proper, $e$ is an
    equivariant closed immersion, and $p$ maps the open complement $Y
    \setminus D$ isomorphically onto $X \setminus C$;
    
\item
a {\em Zariski square} if $p$ and $e$ form an open cover of $X$;
\item 
a {\em $cdh$ square} if it is either a Zariski or an abstract blow-up square. 
\end{enumerate}
The $cdh$ topology on $\cMpctf$ is the topology generated by the
$cdh$ squares. It is a bounded, complete and regular $cd$ structure
by \cite[Theorems 12.7 and 12.8]{chww-monoid}.
\end{defn}

A presheaf of spectra $\cF$ on $\cMpctf$ satisfies the
{\em Mayer-Vietoris property} for some family $\cC$ of cartesian squares
if $\cF(\emptyset)=\ast$ and the application of $\cF$ to each member of the
family gives a homotopy cartesian square of spectra.

\begin{defn}\label{cdh-descent}
Let $\cF$ be a presheaf of spectra on $\cMpctf$. We say that $\cF$
satisfies {\em $cdh$ descent} on $\cMpctf$ if the canonical map
$\cF(X)\to\bH_{\cdh}(X,\cF)$ is a weak equivalence of spectra
for all $X$. Here $\bH_{\cdh}(-,\cF)$ is the fibrant replacement of
$\cF$ in the model structure of \cite{JardineSPS} and
\cite{JardineGen}. By \cite[Proposition 12.10]{chww-monoid}, this is
equivalent to $\cF$ satisfying the Mayer-Vietoris property for
the $cdh$ structure; i.e., for both
the family of Zariski squares and the family of abstract blow-up squares.
\end{defn}

Here is a restatement of \cite[Definition 13.8]{chww-monoid}.

\begin{defn}\label{defn:weakcdh}
Let $k$ be a commutative ring and let $Sch/k$ be the category of
separated $k$-schemes, essentially of finite presentation. 
A presheaf of spectra $F$ on $Sch/k$ is said to satisfy
{\em weak $cdh$ descent} if it satisfies the Mayer-Vietoris property
for all open covers, finite abstract blow-ups, and blow-ups along
regularly embedded subschemes.
\end{defn}

\begin{subex}\label{ex:KH-weakcdh}
The presheaf $\cKH$ satisfies weak $cdh$ descent.  This was
observed in \cite[Example 13.11]{chww-monoid}; the hypothesis there
that the ring $k$ be Noetherian is not needed.
Indeed, the Mayer-Vietoris property for finite abstract blow-ups
(excision for ideals and invariance under nilpotent extensions) is proved in
\cite{WeibelKH} for general rings, and Mayer-Vietoris for open covers as well
as Mayer-Vietoris for blow-ups along regular sequences applies to the category
of quasi-separated and quasi-compact schemes.

We will see other examples
in Theorems \ref{thm:bigdescent} and \ref{E_L} below.
\end{subex}

\begin{thm} \label{thm:descent}
Let $k$ be a commutative regular  ring of finite Krull dimension, with infinite residue fields,
and $F$ a presheaf of spectra on $Sch/k$.
If $F$ satisfies weak $cdh$ descent on $Sch/k$, then the presheaf
$\cF(X) = F(X_k)$ satisfies $cdh$ descent on $\cMpctf$.
\end{thm}

\begin{proof}
The proof is almost identical to the proof of Theorem 14.3 in our 
earlier paper \cite{chww-monoid}, which assumed that $k$ contains a field.

We merely point out the adjustments necessary for the proof to work for $k$.
Although the proof in {\it loc.\,cit.} makes no direct use
of the hypothesis that $k$ contains a field, this hypothesis
is buried in the references to Lemma 13.9 and Theorem 14.2 of
\cite{chww-monoid}. (Theorem 13.3 and Proposition 13.6 are referred to,
but apply as stated).

Although the hypothesis of Lemma 13.9 is that $k$ is a commutative
regular domain containing an infinite field, 
the proof goes through if we only assume
that $k$ is a commutative regular  ring of finite Krull dimension
such that every residue
field of $k$ is infinite.  The hypothesis that $k$ is
regular is needed so that $X_k$ is a Cohen-Macaulay scheme for every
toric monoid scheme $X$ (by \cite{Hochster}).  The finite Krull dimension 
hypothesis is needed in Lemma 13.9 in order that the
local-global spectra sequence converges (and so the proof of Lemma
13.9 is flawed as written; see the remark below for the correction). 

The last hypothesis (which holds in our setting since we assume every residue
field of $k$ is infinite) is needed to use Lemma ~13.7. 

Finally, the hypothesis in Theorem 14.2 that $k$ contains a field is
required to make use of the Bierstone-Milman theorem
\cite[Theorem 14.1]{chww-monoid}.  Replacing it by the more general Theorem
\ref{thm:ros} in this paper makes the proof go through.
\end{proof}

\begin{subrem} 
As explained in the proof above, the hypothesis that $k$ have  finite Krull dimension
is missing in Lemma~13.9 and Theorem 14.3 of \cite{chww-monoid}. 
This however does not affect the main results of
\cite{chww-monoid}, since $K$-theory commutes with filtering colimits
and, by Popescu's theorem \cite{popescu}, every regular ring containing a
field is a filtering colimit of regular rings of finite Krull dimension. 
\end{subrem}

\begin{thm}\label{thm:bigdescent}
Let $k_0$ be a commutative regular ring of finite Krull dimension
all of whose residue fields are infinite,
let $i:k_0\subset k$ be a flat extension of commutative rings.
If $F$ is a presheaf of spectra on $Sch/k$
satisfying weak $cdh$ descent on $Sch/k$, then the presheaf
$\cF_k$ on $\cMpctf$ defined by $\cF_k(X) = F(X_k)$   satisfies $cdh$ descent. 
\end{thm}

\begin{proof}
Let $i_*F$ denote the direct image presheaf on $Sch/k_0$,
defined by $i_*F(S)=F(S_k)$.
Since the flat 
basechange $i_*$ preserves open immersions, closed immersions,
surjective morphisms, finite morphisms, regular closed immersions, and
blow-ups, it also preserves open covers, finite abstract blow-ups, and
blow-ups along regular closed immersions.
Because $F$ satisfies weak $cdh$ descent on $Sch/k$,
$i_*F$ satisfies weak $cdh$ descent on $Sch/k_0$.
The result is now immediate from Theorem \ref{thm:descent}.
\end{proof}

Given a commutative ring $k$,
let $\cK_k$ and $\cKH_k$ denote the presheaves of spectra on $\cMpctf$
sending $X$ to $K(X_k)$ and $KH(X_k)$, respectively.
Here is the analogue of \cite[Corollary 14.5]{chww-monoid}.

\begin{cor}\label{cor:khdescent}
Let $k_0$ be a commutative regular ring of finite Krull dimension
all of whose residue fields are infinite, 
and let $k$ be a commutative flat 
$K$-regular $k_0$-algebra.  Then the presheaf of spectra 
$\cKH_k$ satisfies $cdh$ descent on $\cMpctf$, and the maps
\[
\cKH_k(X) \to \bH_{\cdh}(X, \cKH_k) \leftarrow \bH_{\cdh}(X, \cK_k)
\]
are weak  equivalences for all $X$ in $\cMpctf$.
\end{cor}

\begin{proof}
As observed in Example \ref{ex:KH-weakcdh},
the presheaf $\cKH$ on $Sch/k$ satisfies weak $cdh$ descent. 
Thus the first statement follows from Theorem \ref{thm:bigdescent}, 
and therefore the left arrow above is an equivalence.

For the other arrow, we mimick the argument we used in
\cite[Example 13.11]{chww-monoid}.
By \cite[Theorem 11.1]{chww-monoid},  every $cdh$ covering in
$\cMpctf$ admits a refinement consisting of smooth monoid schemes. If $X$ is smooth, 
then $X_k$ is locally the spectrum of a Laurent
polynomial ring over $k$. Since $k$ is assumed to be $K$-regular, 
we conclude that for smooth $X$, $\cK_k(X)\to \cKH_k(X)$ is an 
equivalence, using Mayer-Vietoris for open covers and the Fundamental Theorem
of $K$-theory.  The result for general $X$ in $\cMpctf$ now follows from the
Mayer-Vietoris property for $\bH_{\cdh}(X,\cK_k)$ and resolution of
singularities for monoid schemes.  
\end{proof}

\section{Presheaves of spectra and dg categories} \label{sec:dg} 

Let $k_0$ be a commutative ring and 
let $E$ be a functor from small dg $k_0$-categories to spectra. We may use $E$ in two different ways to obtain a functor on $\Sch/k_0$. First, 
by regarding a $k_0$-algebra as a dg category with just one object, 
we may restrict $E$ to a functor of commutative $k_0$-algebras. This restriction induces a presheaf $\cE$ on $\Sch/k_0$ by 
mapping $S\mapsto E(\cO(S))$ on $\Sch/k_0$, and  
its fibrant replacement is $\mathbb{H}_{\zar}(-,\cE)$. 
On the other hand, we may simply compose $E$ with the functor that sends a scheme $S$ to the dg $k_0$-category 
$\perf(S)$ of perfect complexes on $S$, obtaining the functor $E(\perf(-))$; 
see \cite[Example 2.7]{chsw} or \cite[Section 2.4]{RC} for a precise definition. 

Following Sections ~3 and ~5 of \cite{Keller}, we say $E$ is 
dg \emph{Morita invariant} if it sends dg Morita equivalences  
to weak equivalences; we say that $E$ \emph{localizing} if 
it sends short exact sequences of dg categories to fibration sequences. 

\begin{lem} \label{lem39}
If $E$ is a functor from small dg $k_0$-categories to spectra that is dg Morita invariant and localizing, then  
the functors $\mathbb{H}_{\zar}(-,\cE)$ and $E(\perf(-))$ are equivalent.
\end{lem}

\begin{proof}
For a commutative ring  $R$, the functor $R \to \perf(R)$ is a dg Morita equivalence and thus
$E(R)\map{\sim} E(\perf(R))$ is an equivalence. It follows that
the natural transformation of presheaves on $\Sch/k_0$ from $E(\cO(-))$ to $E(\perf(\cO(-))$ 
is an equivalence Zariski locally. The induced functor on fibrant
replacements is thus also an equivalence.  Since $E$ is localizing, $S\mapsto E(\perf(S))$ satisfies Zariski descent (and even Nisnevich descent) by 
\cite[Theorem 3.1]{Tabuada}. 
We thus get a pair of natural equivalences
$$
\mathbb{H}_{\zar}(S, \cE) \map{\sim}  \mathbb{H}_{\zar}(S, E(\perf(-))) \lmap{\sim}  E(\perf(S))
$$
for all $S \in \Sch/k_0$.
\end{proof}

From now on, if $S\in\Sch/k_0$ and $E$ is a functor from small 
dg $k_0$-categories that is dg Morita invariant and localizing, 
by $E(S)$ we shall always mean $E(\perf(S))$.

\goodbreak
\smallskip

Given two small dg $k_0$-categories $\cA$
and $\cB$, we write $\cA\oo_{k_0}\cB$ for their
dg tensor product (as defined in \cite[Sec.\,2.3]{Keller}).

\begin{thm}\label{E_L}
Let $k_0$ be a commutative ring, $E$ a functor from small
dg $k_0$-categories to spectra and $\Lambda$ a flat 
(not necessarily commutative) $k_0$-algebra. Assume
\begin{enumerate}
\item $E$ is dg Morita invariant and
\item $E$ is localizing.
\end{enumerate}
Then the presheaf $E_\Lambda$ of spectra on $Sch/k_0$, sending $S$ to
$$
E_\Lambda(S) = E(\perf(S)\otimes_{k_0} \Lambda),
$$
satisfies the Mayer-Vietoris property with respect
to open covers (Zariski descent) and blow-ups of regularly embedded subschemes. 

Assume furthermore that
\begin{enumerate}
\item[(3)] the restriction of $E$ to $k_0$-algebras satisfies excision for ideals 

and is invariant under nilpotent ring extensions.
\end{enumerate}
Then $E_\Lambda$ satisfies weak $cdh$ descent on $\Sch/k_0$. In particular, by Theorem \ref{thm:descent}, 
the functor $\cE_\Lambda(X)=E_\Lambda(X_{k_0})$ satisfies $cdh$ descent on $\cMpctf$.
\end{thm}

\begin{proof}
Because $\Lambda$ is flat, tensoring with $\Lambda$ preserves 
dg Morita equivalences and short exact sequences of
dg categories.
Hence $E_{\Lambda}$ is both dg Morita invariant and localizing whenever $E$ is. In particular it satisfies Zariski descent, by the discussion above. The Mayer-Vietoris property for blow-ups along
regularly embedded subschemes also follows from hypotheses (1) and (2) using \cite[Lemma 1.5]{chsw}, as pointed out in \cite[Theorem 3.2]{CRC16} and implicitly in \cite[Theorem 2.10]{chsw}. To finish the proof we must show that if $E$ also satisfies (3) then it has the Mayer-Vietoris property for finite abstract blow-ups. As observed in the proof of \cite[Theorem 3.12]{chsw}, the latter property follows from excision for ideals,
invariance under nilpotent ring extensions and Zariski descent.  
Tensoring with $\Lambda$ preserves Milnor squares of $k_0$-algebras; 
because $E$ satisfies excision for ideals, so does $E_\Lambda$.
If $I$ is a nilpotent ideal, so is $I\oo_{k_0}\Lambda$, so hypothesis (3) 
implies that $E_\Lambda$ is invariant under nilpotent ring extensions.
\end{proof}

\begin{subrem}\label{X_L}
If $E$ is a dg Morita invariant,
localizing functor from small dg $k_0$-categories to spectra,
and $k$ is a flat commutative $k_0$-algebra, then $E(S_k)$ is naturally
equivalent to the $E_k(S)$ of Theorem \ref{E_L}.  Indeed, the result holds
when $S=\Spec(R)$ is an affine scheme, since
$R_k\to\perf(R_k)$ and $\perf(R)\oo_{k_0}k\to \perf(R_k)$ are 
dg Morita equivalences; the result for arbitrary schemes follows from the 
fact that, by Theorem \ref{E_L}, both $E_k$ and $E(-_k)$ satisfy 
Zariski descent. In particular, it follows that the functors $\cE(X)=E(X_k)$
and $\cE_k(X)=E_k(X_{k_0})$ on $\cMpctf$ are naturally equivalent. 

Motivated by all this, we shall often write $E(X_\Lambda)$ for
$\cE_{\Lambda}(X)$ when $\Lambda$ is a flat but not necessarily
commutative $k_0$-algebra, even if no scheme $X_\Lambda$ is
defined. By Lemma \ref{lem39}, $E(X_\Lambda)$ is equivalent to the
functor defined by Zariski descent from the spectra $E(\Lambda[A])$ on
the affine opens $\MSpec(A)$ of $X$.
\end{subrem}

\goodbreak
\begin{ex}\label{ex:KHC}
For any dg category $\cA$, abusing notation a bit, we also write $\perf(\cA)$ for the (unenriched) category
whose morphisms are $\cA$-module homomorphisms; i.e., the category whose hom sets are the zero cycles of underlying dg category.
Then 
$\perf(\cA)$ may be regarded as a Waldhausen category;
the weak equivalences are quasi-isomorphisms and the cofibrations
are $\cA$-module homomorphisms which admit retractions as homomorphisms 
of graded $\cA$-modules; see \cite[Sec.\,5.2]{Keller} for example.

We define $K(\cA)$ to be the (Waldhausen) $K$-theory of $\perf(\cA)$.
As pointed out in {\it loc.\ cit.,} this functor is localizing and 
Morita invariant, and its restriction to $k_0$-algebras is naturally
equivalent to the usual $K$-theory of algebras.
Theorem \ref{E_L} applies to $KH$, which satisfies (3), showing that
$KH_\Lambda(S)$ satisfies $cdh$ descent on $Sch/k_0$, and
$\cKH_\Lambda(X)$ satisfies $cdh$ descent on $\cMpctf$.
\end{ex}

\begin{cor}\label{cor:khCdescent}
Let $k_0$ be a commutative regular ring of finite Krull dimension
all of whose residue fields are infinite,
and let $\Lambda$ be a flat  
$K$-regular associative $k_0$-algebra.  Then the presheaf of spectra 
$\cKH_\Lambda$ satisfies $cdh$ descent on $\cMpctf$, and the maps
\[
\cKH_\Lambda(X) \to \bH_{\cdh}(X, \cKH_\Lambda) \leftarrow \bH_{\cdh}(X, \cK_\Lambda)
\]
are weak  equivalences for all $X$ in $\cMpctf$.
\end{cor}

\begin{proof}
The left arrow is a weak equivalence by Example \ref{ex:KHC}.
The proof that the other arrow is an equivalence is exactly like
the corresponding proof of Corollary \ref{cor:khdescent}.
\end{proof}

\begin{ex}\label{ex:HC}
The Hochschild homology, cyclic homology and negative cyclic homology
of a dg category $\cA$ 
are defined using the mixed complex of 
$\perf(\cA)$; see \cite[Sec.\,5.3]{Keller}.  Moreover, the restrictions
of these functors to flat algebras (resp., schemes) are isomorphic 
to the usual homology theories of algebras (resp., of schemes).
\end{ex}

\begin{notation}
  Given a presheaf of spectra $\cF$ on some category we will write
  $\cF\oo\Q$ for the rationalization of $\cF$ and $\cF/n$ for the
  cofiber of multiplication by $n$ on $\cF$. We make an exception for
  the presheaf of spectra $HN(-_\Lambda)$ on $\cMpctf$ sending $X$ to
  negative cyclic homology $HN(X_\Lambda)$. By $HN_{\Lambda\Q}(X)$ we
  will mean $HN(X_{\Lambda\otimes\Q})$, which is defined as in Remark
  \ref{X_L}. The Jones-Goodwillie Chern character is a natural
  morphism on algebras, from $K(\Lambda)\otimes\Q$ to
  $HN({\Lambda\otimes\Q})$; see \cite[p.\,351]{Goodwillie86}. It
  extends to a natural transformation of functors on dg categories,
  and hence a natural morphism $\cK_\Lambda(X)\oo\Q \to
  HN_{\Lambda\Q}(X)$ on monoid schemes (see \cite[Example 9.10]{cita},
  \cite[Section ~4.4]{McChern} for the Chern character for dg
  categories).
\end{notation}

\begin{prop}\label{Qfiber}
Let $k_0$ be commutative regular ring of finite Krull dimension all of whose residue
fields are infinite, and let $\Lambda$ be a flat, associative $k_0$-algebra that  is
$K$-regular.  For any monoid scheme $X$ in $\cMpctf$, the following square
of spectra is homotopy cartesian:
\[\xymatrix{
\cK_\Lambda(X)\oo\Q \ar[r] \ar[d] & \cKH_\Lambda(X)\oo\Q \ar[d] \\
HN_{\Lambda\Q}(X) \ar[r]         & \bH_{\cdh}(X,HN_{\Lambda\Q}).
} \]
\end{prop}

\begin{proof}
By Corollary \ref{cor:khCdescent} the presheaves in the right column satisfy
$cdh$-descent on $\cMpctf$. Thus to prove the assertion, we need to show that
the homotopy fiber $F$ of the left vertical map satisfies $cdh$ descent. 
Since $\cK_\Lambda\oo\Q$ and $HN_{\Lambda\Q}$ satisfy Zariski descent, 
so does $F$.  Given Examples \ref{ex:KHC} and \ref{ex:HC},
Theorem \ref{E_L} shows that
it suffices to show that $F$ (the corresponding
homotopy fiber on $Sch/k_0$) satisfies 
excision for ideals and invariance under nilpotent 
extensions.  The
first of these properties is the main theorem of \cite{CortKABI}; the
second is the scheme version of Goodwillie's \cite{Goodwillie86}. 
\end{proof}

The analogue of the Jones-Goodwillie Chern character in characteristic $p$
is the cyclotomic trace; it is a compatible family of morphisms
$tr^\nu:\cK(\Lambda)\to TC^\nu(\Lambda;p)$, where $\Lambda$ is an
associative ring and the pro-spectrum
$\{TC^\nu(\Lambda;p)\}_{\nu}$ is $p$-local topological cyclic homology.
For each $n\ge1$, the cyclotomic trace induces a map
$\cK(\Lambda)/p^n\to TC^\nu(\Lambda;p)/p^n$.
Letting $\nu$ vary, we get a strict map of pro-spectra.
Blumberg and Mandell showed in \cite[1.1, 1.2, 7.1, 7.3]{BlMa2008}
that $TC^\nu$ extends to a functor on dg categories that is
localizing, Morita invariant and satisfies Zariski descent on schemes.

As shown in \cite{BlMa2008} the cyclotomic trace extends to a natural transformation of functors on dg categories; in particular, if $X$ is a monoid scheme and $\Lambda$ is an associative, flat algebra over some commutative ring $k_0$, we obtain a natural map
$tr^\nu:\cK(X_\Lambda)/p^n\to TC^\nu(X_\Lambda;p)$.

Recall from \cite{GH2010} (or \cite[Section 14]{chww-monoid})
that a strict map $\{ E^\nu\}\to \{ F^\nu\}$  of pro-spectra
is said to be a {\it weak equivalence} if for every $q$ the induced map
$\{ \pi_q(E^\nu)\} \to \{\pi_q(F^\nu)\}$
is an isomorphism of pro-abelian groups.
Let 
\begin{equation}\label{square}
\xymatrix{
\{E^\nu\}\ar[r]\ar[d]&\{F^\nu\}\ar[d]\\
\{G^\nu\}\ar[r]&\{H^\nu\}
}
\end{equation}
be a square diagram of strict maps of pro-spectra. We say that \eqref{square}
is {\it homotopy cartesian} if the canonical map from the upper left
pro-spectrum to the level-wise homotopy limit of the other terms is a weak
equivalence.

\goodbreak
The presheaves of pro-spectra we shall consider come from pro-presheaves 
of spectra, that is, from inverse systems of presheaves of spectra. 
Let $\{F^\nu\}$ be a pro-presheaf of spectra on $Sch/k$; for $X$ in 
$\cMpctf$, write $\cF^\nu(X)$ for $F^\nu(X_k)$.

The pro-analogue of $\{F^\nu\}$ (or $\{\cF^\nu\}$)
having the Mayer-Vietoris property for a family of squares 
is the obvious one, as are the notions \ref{cdh-descent},
\ref{defn:weakcdh} of $\{F^\nu\}$ having $cdh$ descent or weak $cdh$ descent.
Here is the pro-analogue of Theorem \ref{thm:descent}.

\begin{thm}\label{pro-descent}
  Let $\{F^\nu\}$ be a pro-presheaf of spectra on $Sch/k$, where $k$
  is a commtative regular ring of finite Krull dimension all of whose
  residue fields are infinite.  If $\{F^\nu\}$, regarded as a presheaf
  of pro-spectra, satisfies weak $cdh$ descent on $Sch/k$, then the
  presheaf $\{\cF^\nu\}$ defined by $\cF^\nu(X)=F^\nu(X_k)$ satisfies
  $cdh$ descent on $\cMpctf$.
\end{thm}

\begin{proof}
We modify the proof of Theorem \ref{thm:descent} (which is in turn
a modification of the proof of \cite[14.3]{chww-monoid}).
By construction, each $\bH_\cdh(-,\cF^\nu)$ satisfies $cdh$ descent.
As noted in the proof of Theorem \ref{thm:descent},
the proof of Theorem 13.9 in \cite{chww-monoid} goes through to show that
the presheaf of pro-spectra $\{\cF^\nu\}$ has the Mayer-Vietoris property for
``nice'' blow-up squares (Definition 13.5 in \cite{chww-monoid}).
Now each $\cF^\nu$ has the Mayer-Vietoris property for Zariski
squares and smooth blow-up squares, so by \cite[12.13]{chww-monoid}
each $\cF^\nu$ satisfies $scdh$ descent (terminology of
\cite[12.12]{chww-monoid}).
Hence each $\{\cF^\nu(X)\}\to \{\bH_\cdh(X,\cF^\nu)\}$ 
is a weak equivalence of pro-spectra 
by Theorem 14.2 of \cite{chww-monoid}.
\end{proof}

\begin{cor}\label{pro-E_L}
Assume that $k_0$ is a commutative regular ring of finite Krull dimension
all of whose residue fields
are infinite, and that $\Lambda$ is a flat unital associative
$k_0$-algebra. Let $\{E^\nu\}$ be an inverse system of functors from the category of small
dg $k_0$-categories to spectra such that each $E^\nu$ is
localizing and dg Morita invariant. 

If the restriction of $\{E^\nu\}$ to $k_0$-algebras, regarded as a functor to pro-spectra, satisfies excision for
ideals and is invariant under nilpotent ring extensions, then
the presheaf of pro-spectra $\{E^\nu_\Lambda\}$, defined by
\[
E^\nu_\Lambda(S) = E^\nu(\perf(S)\otimes_{k_0} \Lambda),
\]
satisfies weak $cdh$ descent on $Sch/k_0$.

By Theorem \ref{pro-descent}, the functor $\{\cE^\nu_\Lambda\}$
satisfies $cdh$ descent on $\cMpctf$.
\end{cor}

\begin{proof}
The proof of Theorem \ref{E_L} applies, using Theorem \ref{pro-descent}
in place of Theorem \ref{thm:descent}.
\end{proof}

\goodbreak
\begin{prop}\label{pfiber}
  Let $k_0$ be a commutative regular ring of finite Krull dimension
  all of those residue fields are infinite, and let $\Lambda$ be a
  flat associative $k_0$-algebra that is $K$-regular.  For any monoid
  scheme $X$ in $\cMpctf$, any prime $p$ and all $n>0$ the following
  square of strict maps of pro-spectra is homotopy cartesian.
\[\xymatrix{
\cK_\Lambda/p^n(X) \ar[r] \ar[d] & \cKH_\Lambda/p^n(X) \ar[d] \\
\{TC^\nu(X_\Lambda;p)/p^n\}_{\nu} \ar[r] & 
\{\bH_\cdh(X_\Lambda,TC^\nu(-;p)/p^n)\}_{\nu}.}
\]
The vertical maps are induced by the cyclotomic trace.
\end{prop}

\begin{proof}
As in Proposition \ref{Qfiber}, 
it suffices to prove that the homotopy fiber $\{\cF^\nu_\Lambda(X)\}$
of the left vertical map satisfies $cdh$ descent as a 
functor on $\cMpctf$, in the sense that 
$\{F^\nu\}\to\{\bH_\cdh(-;\cF^\nu)\}$
is a weak equivalence of pro-spectra.
(This uses the fact that pro-abelian groups form an abelian category.)

For each dg category $\cA$, let $F^\nu(\cA)$ denote the homotopy fiber of 
$K/p^n(\cA)\to TC^\nu(\cA;p)/p^n$.
Each $F^\nu$ is localizing and dg Morita invariant; these properties are
inherited from $K/p^n$ and $TC^\nu(-;p)/p^n$.
To see that $\{F^\nu_\Lambda\}$ 
satisfies weak descent on $Sch/k$,
and hence that $\{\cF^\nu_\Lambda\}$ satisfies $cdh$ descent on
$\cMpctf$, it suffices by Corollary \ref{pro-E_L} to show that
the restriction of the pro-presheaf $\{F^\nu\}$ to algebras, regarded as a functor to pro-spectra, satisfies 
excision for ideals and
invariance under nilpotent ring extensions.
Excision for ideals is \cite[Theorem 1]{GH-birel}; if $I\triqui R$ is an ideal and $f:R\to S$ is a ring homomorphism mapping $I$ bijectively to an ideal of $S$, then the map  $\{F^\nu(R,I)\}\to \{F^\nu(S,f(I))\}$ is a weak equivalence. 
Invariance under nilpotent ring extensions follows from McCarthy's theorem 
\cite{McCarthy97} as strengthened in \cite[Theorem 2.1.1]{GH-dvr}:
if $I$ is nilpotent then $\{F^\nu(R,I)\}$ is weakly equivalent to a point.
\end{proof}

\goodbreak
\section{$\tOmega$ and dilated cyclic homologies.}\label{sec:dilatetc}

Here we briefly recall the results of Section 6 of our paper \cite{chww-p}.  These results do not require 
the existence of a base field.

The {\it cyclic bar construction} $\Ncy(A)$ of a pointed monoid $A$
is the cyclic set whose underlying set of $n$-simplices is
$A\smsh\cdots\smsh A$ ($n+1$ factors), with $t_n$ being rotation
of the entries to the left. It is $A$-graded: for $a\in A$,
a simplex $(\alpha_0, \dots, \alpha_n)$ is in $\Ncy(A;a)$ if
$\prod \alpha_i=a$.
In \cite[3.1]{chww-p}, we introduced the {\it dilated} cyclic bar
construction
\[
\tNcy(A) = \bigvee_{a \in A} \Ncy(A\ad{a}; a),
\]
where the element $a$ in $\Ncy(A\ad{a};a)$ refers to the element $\frac{a}{1}$
of $A\ad{a}$. Thus an $n$-simplex of $\tNcy(A)$ is given by
$(a; \alpha_0, \dots, \alpha_n)$ where $a \in A$,
$\alpha_0, \dots, \alpha_n \in A \ad{a}$,
and the equation $\alpha_0 \cdots \alpha_n= \frac{a}{1}$ holds in $A\ad{a}$.

The geometric realizations of both $\Ncy(A)$ and $\tNcy(A)$ are
{\it $\S$-spaces}, meaning they  have a continuous action of
the circle group $\S$.

We now assume that $A$ is a quotient $\tilde{A}/I$ of a 
cancellative monoid $\tilde{A}$ (i.e., $A$ is {\it partially cancellative}),
and that $A$ is {\it reduced} in the sense
that $a^n=0$ implies $a=0$ (this is equivalent to the assertion that
if $a,b\in A$ satisfy $a^n=b^n$ for all $n>1$ then $a=b$;
see \cite[1.6]{chww-p}).
Such a monoid is naturally contained in a monoid $A_\sn$ that is
{\it seminormal}, meaning that whenever $x,y\in A_\sn$ satisfy $x^3=y^2$
then there is a (unique) $z\in A_\sn$ such that $x=z^2$ and $y=z^3$.
In fact, $A\to A_\sn$ is universal with respect to maps from $A$ to
seminormal monoids; see \cite[Proposition 1.15]{chww-p}.

\goodbreak
\begin{defn}\label{def:tOmega-A} (\cite[4.1]{chww-p})
For a  pc monoid $A$, we define $\tOmega_A$ to be the $\S$-space
\[
\tOmega_A = |\tNcy(A_\sn)|.
\]
Since $A \mapsto A_\sn$ is a functor, the assignment 
$(X,\cA)\mapsto\tOmega_{\cA}=\tOmega_{\cA(X)}$
yields a presheaf on $\cMpctf$. Moreover, there is a natural map
$|\Ncy(A)| \to \tOmega_A$ of $\S$-spaces.
\end{defn}

Given a sequence $\fc = (c_1, c_2, \dots)$ of integers with $c_i \ge 2$ for all $i$, 
it is shown in \cite[3.6]{chww-p} that $|\Ncy (A)|^\fc\map{\simeq}|\tOmega_A|^\fc$
is an $\S$-homotopy equivalence.

The $\S$-equivariant smash product $\tOmega_{\cA}\smsh T$ of
$\tOmega_{\cA}$ with any $\S$-spectrum $T$ is
a presheaf of $\S$-spectra on $\cMpctf$.
For every integer $r\ge1$, we write $\tOmega^{T,r}$ for the
presheaf of fixed-point spectra $X \mapsto (\tOmega_{\cA(X)}\smsh T)^{C_r}$
on $\cMpctf$ (where $C_r$ is the cyclic subgroup of $\S$ having $r$ elements), and 
$\bH_\zar(-, \tOmega^{T,r})$ for its
fibrant replacement for the Zariski topology.
 The following result was proven in \cite[Theorem 6.2]{chww-p}.

\begin{thm} \label{MT2p}
For  any $\S$-spectrum $T$ and integer $r\ge 1$, the presheaf of spectra
$\bH_\zar(-, \tOmega^{T,r})$
satisfies $cdh$ descent on $\cMpctf$.
\end{thm}

Hesselholt and Madsen proved in \cite[Theorem 7.1]{HM} that
for any associative ring 
$\Lambda$ and monoid $A$, there is a 
natural equivalence of cyclotomic spectra,
\[
TH(\Lambda) \smsh |\Ncy(A)| \map{\simeq} TH(\Lambda[A]),
\]
where $TH$ is topological Hochschild homology. 
Combining this with Theorem \ref{MT2p}, we showed in
\cite[Corollary 6.6]{chww-p} that the dilated topological cyclic
homology $X\mapsto TC^\nu(X_k;p)^\fc$
satisfies $cdh$-descent on $\cMpctf$. Here is a more general statement.

\begin{thm} \label{thm:TCdilated}
(\cite[6.6]{chww-p})
Let $\fc = (c_1, c_2, \dotsc)$ be a sequence of integers with
$c_i\geq 2$ for all $i$.  For any associative ring $\Lambda$ and integer $\nu \ge 1$, the
spectrum-valued functor
\[
X \mapsto TC^\nu(X_\Lambda;p)^\fc
\]
satisfies $cdh$ descent on $\cMpctf$.
\end{thm}

\begin{proof}
The proof of Theorem 6.5 and Corollary 6.6 in \cite{chww-p}
go through with $TH(k)$ replaced by $TH(\Lambda)$.
\end{proof}

\goodbreak
\section{Descent for Hochschild and cyclic homology}\label{sec:dilatehn}

We write $HH(\Lambda)$ for the Eilenberg-MacLane spectrum associated to the
absolute Hochschild homology of a ring $\Lambda$ (i.e., relative to the base
ring $\Z$) and similarly for cyclic and negative cyclic homology.
Since the Hochschild complex of $\Lambda$ is a cyclic complex,
$HH(\Lambda)$ is an $\S$-spectrum.

\begin{rem}
The Hochschild homology we are using in this section is the classical
version defined by the bar complex of an algebra.
There is also a derived version (see \cite[IV]{Goodwillie85}), and this is the one that coincides with the Hochschild homology of dg categories used in the rest of this paper.
The results of this section we need going forward (Corollaries \ref{cor:HNdilated} and \ref{cor:FKtorsion}) deal with $\Q$-algebras; for such algebras, the derived and classical definitions agree. 
\end{rem}

\begin{thm}\label{thm:HHdilated}
Let $\fc = (c_1, c_2, \dotsc)$ be a sequence of integers with
each $c_i\geq 2$.  For any $\Z$-flat associative ring $\Lambda$, 
the spectrum-valued functors
\[
X \mapsto HH(X_\Lambda)^\fc
\quad\textrm{ and }\quad
X \mapsto HC(X_\Lambda)^\fc
\]
satisfy $cdh$ descent on $\cMpctf$.
\end{thm}

\begin{proof}
By Theorem \ref{MT2p} with $T=HH(\Lambda)$ and $r=1$,
$\bH_\zar(-,HH(\Lambda)\smsh\tOmega)$ satisfies $cdh$ descent on $\cMpctf$.
Since $|\Ncy|^\fc\map{\simeq}|\tOmega|^\fc$, we see that
$\bH_\zar(-,HH(\Lambda)\smsh|\Ncy|)^\fc$ also satisfies $cdh$ descent.
This is equivalent to $HH(-_\Lambda)^\fc \simeq \bH_\zar(-,HH(-_\Lambda))^\fc$ 
because there is a natural weak equivalence of spectra
\[
HH(\Lambda)\smsh|\Ncy(A)| \simeq HH(\Lambda[A]).
\]
The $HH$ assertion now follows.
The $HC$ assertion follows from this, together with the SBI sequence
connecting cyclic homology and Hochschild homology (see \cite[9.6.11]{WH}).
In more detail, let $\bH\bH$ and $\bH$ denote the presheaves
$\bH\bH(X)=\bH_\cdh(X,HH(-_\Lambda)^\fc)$ and $\bH(X)=\bH_\cdh(X,HC(-_\Lambda)^\fc)$,
and abbreviate $HC_n(X_\Lambda)^\fc$ as $HC_n^\fc$.
By induction on $n$ using the diagram with exact rows
\[\xymatrix{
HC_{n-1}^\fc \ar[r] \ar[d]^{\cong}& HH_n(X_\Lambda)^\fc \ar[r] \ar[d]^{\cong} &
HC_n^\fc \ar[r] \ar[d] & HC_{n-2}^\fc \ar[r] \ar[d]^{\cong} &
\kern-3pt HH_{n-1}(X_\Lambda)^\fc \ar[d]^{\cong} \\
\pi_{n-1}\bH(X) \ar[r] & \pi_n\bH\bH(X) \ar[r] & \pi_n\bH(X) \ar[r]
& \pi_{n-2}\bH(X) \ar[r] & \pi_{n-1}\bH\bH(X),
}\]
we see that the middle map $\pi_nHC(X_\Lambda)^\fc\to\pi_n\bH(X)$ is an isomorphism
for all $n$. 
Hence $HC(X_\Lambda)^\fc\to\bH(X)$ is a homotopy equivalence for all $X$.
\end{proof}
\comment{
In more detail,
Goodwiliie's Theorem (see \cite[II.4.6]{Goodwillie85} or \cite[9.9.1]{WH})
says that if $\Lambda$ is a graded algebra, then the maps
$S: HC_p(\Lambda)^\Q \to HC_{p-2}(\Lambda)^\Q$ are zero.
}

\begin{cor}\label{cor:HNdilated}
Let $\fc = (c_1, c_2, \dotsc)$ be a sequence of integers with
each $c_i\geq 2$.  For any associative $\Q$-algebra $\Lambda$,
the spectrum-valued functor
$$
X \mapsto HN(X_\Lambda)^\fc
$$
satisfies $cdh$ descent on $\cMpctf$.
\end{cor}

\begin{proof}
As in the proof of Corollary 3.13 in \cite{chsw}, the un-dilated presheaf
$HP$ satisfies weak  $cdh$ descent on $Sch/k$. 
By Theorem \ref{thm:bigdescent}, $HP(-_\Lambda)$ satisfies $cdh$ descent 
on $\cMpctf$.  Therefore the dilated presheaf
$X\mapsto HP(X_\Lambda)^\fc$ also satisfies $cdh$ descent on $\cMpctf$.
Using the SBI sequence for $HN$ and $HP$, 
together with Theorem \ref{thm:HHdilated},
it follows that $X\mapsto HN(X_\Lambda)^\fc$ satisfies $cdh$ descent on $\cMpctf$. 
\end{proof}

\begin{cor}\label{cor:FKtorsion}
Let $k_0$ be commutative regular ring of finite Krull dimension all of whose residue
fields are infinite, and let $\Lambda$ be a flat $K$-regular $k_0$-algebra. Then 
\[ K(X_\Lambda)^\fc\otimes\Q \map{\simeq} KH(X_\Lambda)^\fc\otimes\Q.\]
\end{cor}

\begin{proof}
Combine Proposition \ref{Qfiber} and Corollary \ref{cor:HNdilated} 
(the latter applied using the $\Q$-algebra $\Lambda\oo\Q$).
\end{proof}

\goodbreak
\comment{
\smallskip\hrule\smallskip
{\bf Christian's proof sketch:}\\
Periodic cyclic homology is defined by a double complex whose columns are
copies of  the Hochschild complex. After rationalization, the Hodge
decomposition implies that the conditionally convergent spectral  sequence
\[
E^1_{p,q} = H_p(HH^\Q_q(X_k)) \Rightarrow HP^\Q_{p+q}(X_k)
\]
obtained from this unbounded double complex actually
 converges strongly. There is also a strongly convergent spectral
converging to $\pi_{p+q}\bH(X,HP^\Q(-_k))$
and a morphism between these spectral sequences. This morphism
is an isomorphism on the $E^1$ page, by Theorem \ref{thm:HHdilated},
so it induces isomorphisms $HP^\Q_{n}(X_k)\to\pi_n\bH(X,HP^\Q(-_k))$.
Thus $HP^\Q(X_k)\to\bH(X,HP^\Q(-_k))$ is a weak equivalence for all $X$,
as required.
\end{proof}
}

\goodbreak
\section{Main theorem}\label{sec:reduction}

\begin{lem}\label{lem:reducetolocal}
Let $X$ be a monoid scheme and $\fc = (c_1, c_2, \dots)$ a sequence of integers with $c_i \geq 2$ for all $i$.
Let $k$ be a commutative Noetherian ring of finite Krull dimension, and let 
$\Lambda$ be any associative $k$-algebra.
Suppose that for every prime ideal $\wp$ of $k$, the natural map
\[
\phi_{k_{\wp}}:\cK^\fc(X_{\Lambda_\wp}) 
\map{} \cKH^\fc(X_{\Lambda_\wp})
\]
is a weak equivalence. Then $\phi_{k}$ is a weak equivalence.
\end{lem}

\begin{proof}
Write $\cF$ for the presheaf on $S=\Spec(k)$ sending the open 
$U\subset S$ to the homotopy fiber of $\phi_{\cO(U)}$. 
The hypothesis implies that the Zariski sheaves 
$a_\zar\pi_q\cF$ associated to $\pi_q\cF$ vanish on $S$.
Because the source and target, viewed as functors of $U$,
both satisfy Zariski descent \cite[8.1]{TT}, so does $\cF$.
Since $k$ is Noetherian and of finite Krull dimension,
there is a spectral sequence converging to $\pi_*\cF(k)$
with $E_2^{p,q}=H_\zar^p(\Spec k, a_\zar\pi_{-q}\cF)=0$;
see \cite[Theorem V.10.11]{WK}. Hence $\cF(k)\simeq\ast$ and
$\phi_{k}$ is a weak homotopy equivalence.
\end{proof}

\begin{rem} An argument similar to that in the proof of 
Lemma \ref{lem:reducetolocal} shows that 
if $k$ is Noetherian of finite Krull dimension and $\phi_{k^h_{\wp}}$ is a 
weak equivalence for the henselization of every local ring of $k$,
then $\phi_{k}$ is a weak equivalence. One simply has to substitute 
the Nisnevich for the Zariski topology and use \cite[Remark V.10.11.1]{WK}.
\end{rem}

\begin{lem}\label{lem:tower}
Let $(R,\fm)$ be a commutative local ring with finite residue field. 
For each integer $l>0$  there is a finite \'etale extension 
$R \to R'$  of rank $l$ with $(R',\fm R')$ local.
\end{lem}

\begin{proof}
Let $\fm$ be the maximal ideal of $R$ and set  $k=R/\fm$. 
Because $k$ is finite, there is separable field extension $k'/k$ of
degree $l$. Pick a primitive element $\alpha \in k'$ for this
field extension, so that $k'=k(\alpha)$, and let $\bar{f} \in k[x]$ 
be the minimum polynomial of $\alpha$.
Let $f \in R[x]$ be a monic lift of $\bar{f}$ and set $R'=R[x]/f$. 
Then $R \to R'$ is a finite flat extension of rank $l$ and
$R'$ is local because $R'/\fm R' = k'$ is a field.  
Since the field extension is separable, $\bar{f}'(\alpha)\ne0$ 
and hence $f'(x)$ is a unit in $R'$.
This proves $R \to R'$ is an \'etale extension.
\end{proof}

Recall from the introduction that an associative ring $\Gamma$ is (right) regular if it is (right) Noetherian and every finitely generated
(right) module has finite projective dimension. 

\begin{lem}\label{lem:gammasep}
Let $\Gamma$ be an associative algebra over a commutative ring $k$ and $f\in k[x]$ a monic polynomial such that the derivative $df/dx$ is invertible in 
$k'=k[x]/\langle f\rangle$. If $\Gamma$ is regular, then so is $\Gamma'=\Gamma\otimes_kk'$.  
\end{lem}

\begin{proof} 
First of all, $\Gamma'$ is  Noetherian since it is finite as a module over $\Gamma$, which is  Noetherian by assumption.
We must show every finitely generated
right $\Gamma'$-module $M$ has finite projective dimension.
This is clear for $M$ of the form $N\otimes_kk'$ for some
finitely generated $\Gamma$-module $N$, since 
$k\to k'$ is flat and we are assuming that $\Gamma$ is regular. If $M$ is any $\Gamma'$-module, write $M'=M\otimes_kk'$; the multiplication map $\mu:M'\to M$ is
a surjective homomorphism of $\Gamma'$-modules. On the other hand, because $k'$ is a separable $k$-algebra, there exists an idempotent $e=\sum_ix_i\otimes
y_i\in k'\otimes_kk'$such that  $\sum_ix_iy_i=1\in k'$ and such that 
$(x\otimes 1-1\otimes x)e=0$ for all $x\in k'$. One checks that the map $s:M\to M'$, $m\mapsto me$ is $\Gamma'$-linear and that $\mu s=1_M$. Thus $M$ is a
direct summand of $M'$, and a direct summand of a module of finite
projective dimension is itself of finite projective dimension.
\end{proof}

We are now ready to prove our main theorem.

\begin{thm}\label{thm:main}
Let $k_0$ be a commutative regular ring of finite Krull dimension,
$\Lambda$ a flat (not necessarily commutative) $k_0$-algebra, and assume that one of  the following three conditions
holds:
\begin{enumerate}
\item $\Lambda$ is regular,
\item $\Lambda$ is $K$-regular and all of the residue fields of $k_0$ are infinite, or 
\item 
$\Lambda$ is $K$-regular and for every prime $\fp\subset k_0$ whose residue field is finite
and every \'etale $(k_0)_\fp$-algebra $k'$, the ring $\Lambda\oo_{k_0} k'$ is $K$-regular.
\end{enumerate}
Let $X$ be a pctf monoid scheme of finite type, and
$\fc = (c_1, c_2, \dotsc )$ any sequence of integers with
$c_i\ge 2$ for all $i$.
Then the natural map
\[
K^\fc(X_\Lambda) \to KH^\fc (X_\Lambda)
\]
is a weak equivalence.
\end{thm}

\begin{proof}
Both regularity and $K$-regularity localize under 
central nonzero divisors \cite[Lemma~V.8.5]{WK}.
It follows from this and from Lemma \ref{lem:reducetolocal} 
that we may assume that $k_0$ is local. 

Suppose first that the residue fields of $k_0$ are infinite. Let
$\cF_K$ be the homotopy fiber of the map $K\to KH$. We must show that the
homotopy groups of $\cF_K^\fc(X_\Lambda)$ vanish. By Corollary
\ref{cor:FKtorsion}
we have $\pi_*(\cF_K^\fc(X_\Lambda))\otimes\Q=0$. Hence the groups
$\pi_*(\cF_K^\fc(X_\Lambda))$ are torsion. We will be done if we show
that multiplication by any prime $p$ induces an isomorphism on
$\pi_*(\cF_K^\fc(X_\Lambda))$, or equivalently, that
$\pi_*(\cF_K^\fc(X_\Lambda)/p)=0$. This follows from 
Proposition \ref{pfiber} and Theorem \ref{thm:TCdilated}, using
\cite[Lemma 8.2]{chww-p}, as in the proof of \cite[Theorem 8.3]{chww-p}. This proves the result if (2) holds.

Now assume that the residue field $k_0/\fm$ of $k_0$ is finite and that either (1) or (3) hold. 
Let $\cF^\fc(k_0)$ be the fiber of the map $K^\fc(X_\Lambda)\to KH^\fc(X_\Lambda)$.
We must show that the homotopy groups of $\cF^\fc(k_0)$ vanish.
Let $l$ be a prime number. By Lemma \ref{lem:tower} there exists
a tower of finite, \'etale  extensions $k_0\subset k_1\subset\cdots$ of the form $k_{i+1}=k_i[x]/\langle f_i\rangle$ with $df_i/dx$ invertible in $k_{i+1}$,  
such that $(k_n,\fm k_n)$ is local and has rank $l^n$ over $k_0$, for all $n$. 

Set $k'=\bigcup_n k_n$; it is Noetherian by 
\cite[$0_{III}$ 10.3.1.3]{EGAIII}. Then $(k',\fm k')$ is 
regular local, with infinite residue field, and 
$k'\subseteq \Lambda\otimes_{k_0} k'$ is flat. If (3) holds, then 
$\Lambda\otimes_{k_0} k'$ is $K$-regular by hypothesis. If (1) holds, then by Lemma \ref{lem:gammasep}, 
$\Lambda_n=\Lambda\otimes_{k_0}k_n$ is regular and thus $K$-regular for each $n$, whence again we conclude that $\Lambda\otimes_{k_0} k'$ is $K$-regular. Hence 

\[
0=\pi_*\cF^\fc(k')=\colim_n\pi_*\cF^\fc(k_n).
\]
Since $\Lambda \to \Lambda\otimes_{k_0} k_n$ is finite and flat, 
there is a natural transfer map

$K(X_{\Lambda\oo_{k_0} k_n}) \to K(X_\Lambda)$ such that the composition
$K(X_\Lambda) \to K(X_{\Lambda\otimes_{k_0} k_n}) \to K(X_\Lambda)$ 
induces multiplication by $l^n$ on homotopy groups, for all $n$.
By naturality, $K^\fc$ also
admits such a transfer map. Likewise, $KH$-theory and hence
$KH^\fc$-theory have such transfer maps and they are compatible with the
map $K^\fc \to KH^\fc$.  We obtain a map $\cF^\fc(k_n)\to\cF^\fc(k_0)$ such
that the composition $\cF^\fc(k_0) \to \cF^\fc(k_n) \to \cF^\fc(k_0)$ induces
multiplication by $l^n$ on homotopy groups.  It follows that the
kernel of $\pi_*\cF^\fc(k_0) \to \pi_*\cF^\fc(k^\prime)$ is an $l$-primary
torsion group. Since this occurs for every prime $l$,
we must have $\pi_*(\cF^\fc(k_0))=0$.
\end{proof}

\goodbreak

Theorem \ref{thm:main} implies Theorem \ref{intro-Dilation}:

\begin{proof}[Proof of Theorem \ref{intro-Dilation}]
Cases (c) and (e) of Theorem \ref{intro-Dilation} are cases (1) and (3) of Theorem \ref{thm:main}, and case (d) is a special case of (e).
If $k$ is a regular commutative ring and $\Spec(k)$ is connected, 
then either $k$ is flat over a field $k_0$ or 
it is flat over the ring of integers. Hence it satisfies the hypothesis of Theorem \ref{thm:main}. If $k$ is an arbitrary commutative regular ring, then it is a finite product of
regular rings with connected $\Spec$. Thus case (a) of Theorem \ref{intro-Dilation} is proved. 

If $\Lambda$ is a commutative $C^\ast$-algebra, then it is flat over
the field $k_0 = \C$ and satisfies the $K$-regularity hypothesis by
\cite[Theorem 8.1]{CortinasThom}. Thus case (b) of 
Theorem \ref{intro-Dilation} also follows from Theorem \ref{thm:main}.
\end{proof}

\begin{rem}\label{rem:mainnoeth}
If in Theorem \ref{thm:main} we assume that $\Lambda$ is commutative
Noetherian and $K$-regular then every \'etale extension of $\Lambda$ 
is $K$-regular by van der Kallen's theorem \cite[Theorem 3.2]{vdK}.
In particular the assumption that $\Lambda\oo_{k_0}k'$ is $K$-regular 
for every \'etale $k_0$-algebra $k'$ is superfluous in this case.
\end{rem}

\begin{rem}\label{rem:altC*}
Theorems \ref{Main-intro} and \ref{intro-Dilation} for commutative
$C^\ast$-algebras can alternatively be derived from 
Gubeladze's main result of \cite{Gu08} 
using \cite[Theorem 7.7]{CortinasThom}.
\end{rem}

\bibliographystyle{plain}

\end{document}